\documentclass[11pt, a4paper]{amsart}
\usepackage{graphicx} 
\usepackage[utf8x]{inputenc}

\usepackage{amsmath,amsthm,amscd,amssymb,amsfonts,amsbsy}
\usepackage{mathtools} 

\usepackage{cancel, textcomp}
\usepackage[mathscr]{euscript}
\usepackage[nointegrals]{wasysym}

\usepackage{comment}

\usepackage[colorlinks,citecolor=red,pagebackref,hypertexnames=false]{hyperref}
\usepackage{latexsym}

\usepackage{graphicx}
\usepackage{bbm}
\usepackage{enumerate}

\usepackage{tikz-cd}  
\usepackage{color}  
\usepackage{microtype} 

\def\bN{\mathbb{N}}
\def\bC{\mathbb{C}}
\def\bR{\mathbb{R}}

\def\bZ{\mathbb{Z}}
\def\p{\par}
\def\dd{\mathrm{d}}

\theoremstyle{plain}
\newtheorem{theorem}[equation]{Theorem}
\newtheorem{lemma}[equation]{Lemma}
\newtheorem{corollary}[equation]{Corollary}
\newtheorem{proposition}[equation]{Proposition}
\newtheorem{conjecture}[equation]{Conjecture}

\theoremstyle{definition}
\newtheorem{definition}[equation]{Definition}

\newtheorem{problem}[equation]{Problem}

\theoremstyle{remark}

\newcommand{\defn}[1]{\begin{definition}#1\end{definition}}

\allowdisplaybreaks

\numberwithin{equation}{section}

\usepackage{chemfig}

\DeclareMathOperator{\id}{Id}
\DeclareMathOperator{\Aut}{Aut}
\def\eps{\varepsilon}

\begin{document}

\title{Quantitative Preservations of Ulam Stability-type Estimates}
\author{Mason Sharp}
\address{Mason Sharp, Department of Mathematics
\\
University of Alabama
\\
Tuscaloosa, AL, 35487, USA}
\email{mesharp3@ua.edu}
\date{January 26 2025}
\begin{abstract}
  We show some preservation results of amenably extending strongly Ulam stable groups under mild decay assumptions, including quantitative preservation of asymptotic bounds under the assumption that the modulus of stability is H\"older continuous of exponent $s>\frac 1 2$ at 0, utilizing some simplistic integral estimates. Additionally, we show some partial results around inductive preservation of modulus bounds in infinite dimensions using these integral estimates, as well as strong quantitative preservation in the finite dimensional case. This implies the existence of $\mathfrak{U}$ uniformly stable existential closures among groups with sufficiently large Lipschitz estimates of any countable group. Finally, we show quantitative control preserving difficulty of approximation of maps over stable groups on diagonally embedding into higher dimensions.
\end{abstract}
\maketitle


\section{Introduction}

\sectionmark{Introduction}
\p Amenability of groups has been a quite useful property in studying unitary representations, or otherwise geometric-flavored and group von Neumann algebra properties about groups and their actions, as evidenced by the extensive literature and litany of equivalent formulations found over the years. For our purposes, we focus on a particular unitary representation stability property held by these groups. 

\begin{definition}
    Given a group $H$, Hilbert space $\mathcal H$, and map $f:H\to U(\mathcal H)$ with $f(1)=\id$, we define the defect $\delta(f)$ by \begin{equation*}\delta(f)=\sup_{x,y}\Vert f(xy)-f(x)f(y)\Vert\end{equation*}
\end{definition}
Using this definition, by an easy triangle inequality argument, we can find
\begin{proposition}\label{prop2}
    If $\Vert f-g\Vert \leq a$, then $\delta(g)\leq \delta(f)+3a$
\end{proposition}
so this particular function is $\infty$-norm continuous. 
\defn{For a group $G$, we define $F_G:[0,2]\to [0,2]$ by 
\begin{align*}
    F_G(\eps) =& \sup \{\inf\{ \Vert f-\rho\Vert_\infty\;\mid\; \rho:G\to U(\mathcal H)\text{ is a representation}\}\\
    &\quad\quad\quad\mid \;\mathcal H\text{ is Hilbert,}f:G\to U(\mathcal H),f(1)=\id,\delta(f)\leq \eps\}.
\end{align*}
}
We note that though Hilbert spaces may be a proper class, these are subclasses of $\bR$ bounded above by 2 and it is sufficient to find any continuous function $[0,2]\to[0,2]$ bounding this from above and evaluating to 0 at 0, so well-definedness of this particular function is less important than studying what bounds it. A further inquiry into bounding the dimensions required, or otherwise ensuring it is always well-defined without nontrivial set theoretic maneuvering, would be a useful endeavor, but we will simply assume it exists as a function for our groups to which the following definition applies.

\defn{A group $G$ is \emph{Strongly Ulam Stable} (SUS) with modulus $F_G$ if the function is continuous at 0.}

\p The relevance of this property is the following theorem of Kazhdan and the motivation behind this work:
\begin{theorem}[Kazhdan \cite{Kaz}]
    If $G$ is amenable, then $G$ is strongly Ulam stable. Furthermore, $\frac 1 2 \eps \leq F_G(\eps) \leq 2\eps$ for sufficiently small $\eps$.
\end{theorem}

\p Utilizing similar techniques to how one can prove this Theorem 1.5, we prove the main result (Theorem \ref{productpreserve.thm}) of section 3 below:

\begin{theorem}
    Suppose $G$ is an amenable group with mean $\int\dd g$ and $H$ is SUS with $F_H(\eps) \in O(\eps)$, then $G\times H$ is SUS with $F_{G\times H} \in O(\eps)$.
\end{theorem}

Further, by the methods and bounds involved in the proof generalizing appropriately, we prove Corollaries \ref{3.13} and \ref{3.14} respectively strengthening this by weakening the estimate to only require $O(\eps^s)$ for some $s>\frac 1 2$, or extending to groups of the form $G \rtimes_a H$ where $a[H]\leq Aut(G)$ is amenable, for which Theorem \ref{productpreserve.thm} corresponds to $a[H] = 1$. In fact, as expounded upon lightly in Remark \ref{3.17}, these proofs are local to each Hilbert space.
\p The case of having a Lipschitz (or, in the phrasing of \cite{FR1}, linear) estimate is of particular relevance and importance as these correspond to vanishing asymptotic cohomology groups, as well as Kazhdan's theorem (Theorem 1.5 above) stating that amenable groups do in fact have such an estimate. It is for this reason that we separated Corollary \ref{3.13} from the proof, and Corollary \ref{3.14} was separated due to the distinct treatment of the variable of integration as opposed to similar minor generalizations like groups of the form $H \rtimes G$ for $G$ amenable. The relevance of these semidirect products can be seen in \cite{Alp23}, as restricted wreath products directly witness amenability by results of Alpeev, though we fail to prove preservation in wreaths due to injectivity of the map into $Aut(G)$. 
\p We note that although \cite{FR1} in Propositions 1.5 and 1.8 does prove an essentially stronger result than \ref{productpreserve.thm} as stated, in that coamenable embeddings preserve stability properties, our result gives a novel method of this case through explicit controls of convergence rates, as well as the novel result that it further applies to nonlinear control in these particular circumstances, with quantitative control on the stability moduli of the extensions in question.

\p In a different direction, we also examine the closure of the class under nice direct limits in the form of inductivity, which we demonstrate for finite dimensions in Theorem \ref{4.6} below:

\begin{theorem}
    Given a proper $\kappa$-sequence of subgroups $(G_i)$ which are all uniformly $(U(d), \Vert\cdot\Vert_s)$ stable for any norm $\Vert\cdot \Vert_s$ such that $F(\eps) = \sup_i F^{(d)}_{G_i}(\eps)$ is continuous at 0, we have that $G = \bigcup_i G_i$ is uniformly stable with $F^{(d)}_G(\eps)\leq F(\eps)$.
\end{theorem}
Here, $F^{(d)}_G$ is defined identically to $F_G$ but only looking at $(U(d),\Vert\cdot\Vert_s)$. The main implication of this is the existence of existential closures among groups with uniform bounds (Corollary \ref{4.9}), though further quantitative control regarding restricted wreaths can provide significant strengthenings.

\p Finally, we present in Section 5 through Proposition \ref{prop5.2} a way one can encode the germ of the SUS modulus $F_G$ under assumption of $O(\eps^s)$ decay for some $s>\frac{1}{2}$. Because of this, and the other results proven here or in \cite{FR1}, \cite{BOT}, and \cite{GLMR}, we posit the following conjecture (Conjecture \ref{conj5.1})
\begin{conjecture}
    Every Strongly Ulam Stable group $G$ has $O(\eps)$ decay at 0 of the SUS modulus $F_G$.
\end{conjecture}

\p To date, no examples of SUS groups which are not amenable have been found, which brings into question if any such groups exist at all. We present partial results towards nonexistence of non-amenable SUS groups by showcasing various meta-structural properties of the class of SUS groups which are shared by amenable groups, giving a finer understanding of how a counterexample must behave.

For further reading on related properties and extended exposition, we refer the reader to \cite{BOT} and \cite{deChi}

\subsubsection*{Acknowledgements} We thank Ilijas Farah, Dalton Sakthivadivel, Aareyan Manzoor, and Anthony Chen for helpful correspondence on ideas for and around this note, as well as Simon Bortz for extended discussion on style and the intuition around estimates.

\section{Integral estimates}
\sectionmark{Integral estimates}
\p In order to manipulate and bound various functions on our groups, it is necessary to develop some tools by which we can transform or combine them. To that end, we suggest integral transforms as integration has relatively nice properties and is well understood, giving us much material to work with. First, we define how such an integral is evaluated.

\defn{If $\mu$ is a finitely additive measure on a set $X$, and $A$ is a normed vector space, denote by $[X]A$ the set of (essentially) bounded measurable functions $X\to A$ under the supremum norm $\Vert \cdot \Vert_\infty$.}

\defn{In the same setting as before, if $A$ has predual $B$, we define for $f_x\in [X]A$ the integral $\int f\dd\mu$ as the linear map $(b\mapsto \int f_x(b)\dd\mu): B\to\bC$ when it exists.}

\p In these above definitions, this creates a linear map on $B$ to $\bC$, up to integrability of the functions at hand. If this linear map is continuous, since $B^* = A$, there is a point in $A$ to which this integral evaluates, so all we need to show is boundedness of this linear map. If we assume that taking the norms of our function is measurable (as the function $x\mapsto \Vert f_x\Vert$, that is), we have the following proposition assuring boundedness. 
\begin{lemma}
    Given appropriate $\mu$, $f_x\in [X]A$, such that the norm $\Vert f_x\Vert$ is $\mu$-measurable, we have that \begin{equation}
    \bigg\Vert\int f_x\dd \mu \bigg\Vert \leq \int \Vert f_x\Vert \dd \mu
\end{equation}
\end{lemma}

\begin{proof}
    Since $A = B^*$, for all $x \in X$ we know $\Vert f_x\Vert =\sup_{b\in B;\Vert b\Vert_B = 1} \vert f_x(b)\vert$, and we note that for any $\Vert b\Vert_B=1$ by definition of the integral we see
    \begin{align*}
        \bigg\vert\bigg(\int f_x\dd\mu\bigg)(b)\bigg\vert &= \bigg\vert\int f_x(b)\dd\mu\bigg\vert\\
        &\leq \int \vert f_x(b)\vert\dd\mu\\
        &\leq \int \Vert f_x\Vert \Vert b\Vert_B\dd\mu = \int \Vert f_x\Vert \dd\mu
    \end{align*}
\end{proof}
Refer to \cite{deChi} for a related statement in the case of amenability with fuller generality and specificity regarding the norms considered.

\p From this bound on norms we clearly see the following, applying Tomiyama's theorem (\cite{Blackadar} II.6.10.2 for the statement used) after identifying $A$ with the space of constant functions and endowing $[X]A$ with the supremum norm.
\begin{lemma}
    If $\mu$ is a finitely additive probability measure on $X$, and $A$ is a Banach space with predual $B$, then $\int\dd\mu: [X]A\to A$ is a conditional expectation.
\end{lemma}

\p In the case of $A = B(\mathcal H)$, since we can determine $f\in A$ by the maps on $\mathcal H^2$ defined by $(x,y)\mapsto \langle fx, y\rangle = \overline{\langle f^*y, x\rangle}$, we can see that $\big\langle \big(\int f_t\dd t\big) x, y\big\rangle = \int \langle f_t x, y\rangle\dd t$. This yields, under a small amount of manipulation, the following useful property:
\begin{lemma}
    For appropriately measurable $f_x$, we have that 
    \begin{equation}
    \bigg(\int f_x\dd\mu\bigg)^*=\int f_x^*\dd\mu
\end{equation}
\end{lemma}

\p The case most important to our purposes is where $X = G$ is an amenable group, which comes endowed with an invariant mean $\int\dd\mu$. This gives use a probability measure $\mu$ on $P(G)$ such that for any $g\in G, F\subset G$ we know that $\mu(F) = \mu(Fg)$. Similarly, for any bounded function $f:G\to \bC$, we have a value $\int f(x)\dd\mu(x) = \int f(xg)\dd\mu(x) \in \bC$. In this case, all our norm functions are measurable, and furthermore integrable so long as they're bounded, hence all bounded functions $G\to A$ are in $[G]A$ with respect to $\mu$.

\section{Product Preservation}
\sectionmark{Product Preservation}
\p We use similar techniques to those used in \cite{BOT, deChi} to prove the following result, which is Theorem 1.6 above: 
\begin{theorem}\label{productpreserve.thm}
    Suppose $G$ is an amenable group with mean $\int\dd g$ and $H$ is SUS with $F_H(\eps)\in O(\eps)$, then $G\times H$ is SUS with $F_{G\times H}\in O(\eps)$
\end{theorem}

\begin{proof}
    Let $\mathcal H$ be Hilbert and $U = U(\mathcal H)$, and let $f:G\times H\to U$ be a function such that $f(1,1)=\id$. We define $f'(x,y)$ for each $(x,y)$ as the integral
    \begin{equation}
        f'(x,y)=\int f(g,1)^*f(gx,y)\dd g \in B(\mathcal H).
    \end{equation}
    Since $\int\dd g$ is a conditional expectation (and, further, behaves as expected with norms), we know that these integrals are bounded in norm by 1 and $\int f(x,y)\dd g = f(x,y)$ for each $(x,y)$, and as such
    \begin{align*}
        \Vert f'(x,y)-f(x,y)\Vert &=\bigg\Vert\int f(g,1)^*f(gx,y) - f(x,y)\dd g\bigg\Vert\\
        &= \bigg\Vert\int f(g,1)^*f(gx,y) - f(g,1)^*f(g,1)f(x,y)\dd g\bigg\Vert\\
        &= \bigg\Vert\int f(g,1)^*(f(gx,y)-f(g,1)f(x,y))\dd g\bigg\Vert\\
        &\leq \int \Vert f(g,1) ^*(f(gx,y)-f(g,1)f(x,y))\Vert\dd g\\
        &\leq \int \Vert f(g,1)\Vert \Vert f(gx,y)-f(g,1)f(x,y)\Vert\dd g\\
        &\leq \int \Vert f(gx,y)-f(g,1)f(x,y)\Vert \dd g\\
        &\leq \delta(f).
    \end{align*}
    Since $\lVert \int (f(gx,y)-f(g,1)f(x,y))^*(f(gx',y')-f(g,1)f(x',y'))\dd g\rVert \leq \delta(f)^2$, and by expanding out this integral we find that
    \begin{align*}
       & \int (f(gx,y)-f(g,1)f(x,y))^*(f(gx',y')-f(g,1)f(x',y'))\dd g\\
       &= \int f(gx,y)^*f(gx',y')-f(gx,y)^*f(g,1)f(x',y')\\& \quad\quad-f(x,y)^*f(g,1)^*f(gx',y')+f(x,y)^*f(g,1)^*f(g,1)f(x',y')\dd g\\
       &=\int f(gx,y)^*f(gx',y')\dd g - \bigg(\int f(g,1)^*f(gx,y)\dd g\bigg)^*f(x',y')\\
       &\quad\quad -f(x,y)^*\int f(g,1)^*f(gx',y')\dd g+f(x,y)^*f(x',y')\\
       &= \int f(gx,y)f(gx'y')\dd g- f'(x,y)^*f(x',y')-f(x,y)^*f'(x',y') +f(x,y)^*f(x',y')^*\\
       &= \int f(gx,y)^*f(gx',y')\dd g - f'(x,y)^*f'(x',y') + (f(x,y)-f'(x,y))^*(f(x',y')-f'(x',y')).
    \end{align*}
    \newline
    So, by triangle inequality, 
    \begin{equation}\label{ineq4}
        \bigg\Vert \int f(gx,y)^*f(gx',y')\dd g - f'(x,y)^*f'(x',y')\bigg\Vert \leq 2\delta(f)^2
    \end{equation}
    as stated in \cite{BOT}, and in the case $y=1$, we recover some of the statements made in that paper immediately via the manipulations which follow.
    
    \p Manipulating \eqref{ineq4} using amenability as $\int \dd g$ is  $G$-invariant, we replace $g$ with $gx^{-1}$ to show $\int f(gx,y)^*f(gx',y')\dd g= \int f(g,y)^*f(gx^{-1}x',y')\dd g$, from which we see
    \begin{align*}
        &\bigg\Vert\int f(gx,y)^*f(gx',y')\dd g -f(1,y)^*f'(x^{-1}x',y')\bigg\Vert \\
        =& \bigg\Vert\int f(g,y)^*f(gx^{-1}x',y')\dd g - f(1,y)^*\int f(g,1)^*f(gx^{-1}x',y')\dd g\bigg\Vert\\
        =& \bigg\Vert\int f(g,y)^*f(gx^{-1}x',y') - f(1,y)^*f(g,1)^*f(gx^{-1}x',y')\dd g\bigg\Vert\\
        =& \bigg\Vert\int (f(g,y)^* - f(1,y)^*f(g,1)^*)f(gx^{-1}x',y')\dd g\bigg\Vert\\
        \leq& \int \Vert (f(g,y) - f(g,1)f(1,y))^*f(gx^{-1}x',y')\Vert\dd g\\
        \leq& \int \Vert (f(g,y) - f(g,1)f(1,y))\Vert \dd g\\
        \leq& \sup_{a,b} \Vert f(a,b) - f(a,1)f(1,b)\Vert \coloneqq \delta_s(f).
    \end{align*}
    \p In the case $x=x',y=y'$, we see via application of the above inequalities and triangle inequality, provided that $\delta(f)$ is sufficiently small, that
    \begin{equation*}
        (1-2\delta(f)^2-\delta_s(f))I \leq f'(x,y)^* f'(x,y) \leq I
    \end{equation*}
    and hence $f'(x,y), |f'(x,y)|$ are invertible (and $\vert f'(x,y)\vert$ also satisfies the same inequality). We define $f_1(x,y) = f'(x,y)\vert f'(x,y)\vert^{-1}$, which is unitary by polar decomposition. By some algebraic manipulations, we see that 
    \begin{align}
        \big\Vert f_1(x,y) - f'(x,y)\big\Vert &= \big\Vert f'(x,y)\vert f'(x,y)\vert^{-1} - f'(x,y)\big\Vert \nonumber\\ 
        &= \big\Vert f'(x,y)\vert f'(x,y)\vert^{-1} - f'(x,y)\vert f'(x,y)\vert^{-1}\vert f'(x,y)\vert\big\Vert \nonumber\\
        &=\big\Vert \big[f'(x,y)\vert f'(x,y)\vert^{-1}\big]\big[I-\vert f'(x,y)\vert\big]\big\Vert\nonumber\\
        &\leq \big\Vert I-\vert f'(x,y)\vert\big\Vert. \label{normalized}
    \end{align}
    \p Since $\langle At,t\rangle$ is linear in $A$ and $\Vert A\Vert = \sup_{\Vert t\Vert \leq 1} \langle At,t\rangle$, we can utilize the inequality \eqref{normalized} directly above to bound the norm by $2\delta(f)^2+\delta_s(f)$.
    \p Combining these bounds, we see that $\Vert f(x,y)-f_1(x,y)\Vert \leq \delta(f)+\delta_s(f)+2\delta(f)^2$. Next, we want to control the defect of $f'$ and $f_1$ in terms of that of $f$. We know some partial bounds for the defect of $f'$ from our above manipulations,
    \begin{equation}
        \Vert f'(x'x,y) - f'(x',1)f'(x,y)\Vert \leq 2 \delta(f)^2
    \end{equation}
    so it suffices to bound some partial defects of $f'$, such as the difference between $f'(1,y)^*$ and $f'(1,y^{-1})$, in order to determine strong bounds on $\delta(f')$. First, we consider the defect localized to $H$.
    \begin{align}
        \Vert f'(1,y) - f(1,y)\Vert &= \bigg\Vert \int f(g,1)^*f(g,y) - f(1,y)\dd g\bigg\Vert\nonumber\\
        &= \bigg\Vert \int f(g,1)^*(f(g,y) - f(g,1)f(1,y))\dd g\bigg\Vert\nonumber\\
        &\leq \int \Vert f(g,1)^*(f(g,y) - f(g,1)f(1,y)) \Vert\dd g\nonumber\\
        &\leq \int \Vert f(g,y) - f(g,1)f(1,y) \Vert\dd g\nonumber\\
        &\leq \delta_s(f) \label{Hbounding1}
    \end{align}
    \p If we define an error term localized to $H$ by
    \begin{equation}
        \delta_H(f) = \sup_{y,y'}\Vert  f(1,yy') - f(1,y)f(1,y')\Vert
    \end{equation}
    We can apply the bound from \eqref{Hbounding1} repeatedly to relate $\delta_H(f')$ to that of $f$, and we arrive at the following helpful estimate:
    \begin{equation}
        \delta_H(f') \leq 3\delta_s(f) + \delta_H(f).
    \end{equation}
    \p The last term to bound is $\delta_c(f') \coloneqq \sup_{x,y}\Vert f'(x,1)f'(1,y)-f'(1,y)f'(x,1)\Vert$. We claim this is sufficient, as for any $f$ we can bound 
    \begin{equation}\label{ineqpartition.eq}
        \delta(f)\leq \delta_G(f)+\delta_H(f)+\delta_c(f)+3\delta_s(f),
    \end{equation}
    where $\delta_G$ is defined analogously to $\delta_H$, which is trivial to verify by repeated triangle inequality applications but provides a very useful partition localizing error behavior.
    
    \p Putting this all together: by a result of Kazhdan \cite{Kaz}, we know that $G$ is SUS with $F_G(\eps) \leq 2\eps \in O(\eps)$, and by assumption $H$ is SUS with $F_H(\eps) \leq c\eps \in O(\eps)$ for all sufficiently small $\eps > 0$. Suppose $f_0:G\times H \to U$ satisfies $f(1,1)=I$ and $\delta(f) = \eps$. Let $X: G\to U$ and $Y_0:H\to U$ be unitary representation approximations guaranteed by SUS, and let ${_0f}(x,y) = Y_0(y)X(x)$. By iterated triangle inequality applications, we see ${_0f}$ is within $(3+c)\eps$ of $f_0$, and hence $\delta({_0f})\leq (10+3c)\eps$ by Proposition \ref{prop2}.
    \p Taking ${_0f}'$ as above, we see this is within $(13+4c)\eps$ of $f_0$. Further, we see that, since $X(gx) = X(g)X(x)$ and $X(g)^* = X(g^{-1})$, we have that
    \begin{align*}
        {_0f}'(x,1){_0f}'(1,y) &= \int X(t)^*Y_0(1)X(tx)\dd t\int X(g)^*Y_0(y)X(g)\dd g\\
        &= \int X(t)^*X(t)X(x)\dd t\int X(g)^*Y_0(y)X(g)\dd g\\
        &= \bigg(\int X(t^{-1}t)\dd t\bigg) X(x)\int X(g)^*Y_0(y)X(g)\dd g\\
        &= \int X(x^{-1})^*X(g)^*Y_0(y)X(g)\dd g\\
        &= \int X(gx^{-1})^*Y_0(y)X(g)\dd g\\
        &= \int X(g)^*Y_0(y)X(gx)\dd g = {_0f}'(x,y).
    \end{align*}
    \p Only the last line relied on amenability, and we see that by similar reasoning to the first line that ${_0f}'(x,1) = X(x)$ and ${_0f}'(x,y) = {_0f}'(1,y)X(x)$. From this, we know that $\delta_G, \delta_s, \delta_c$ all evaluate to $0$ for ${_0f}'$. So, by \eqref{ineqpartition.eq} we know that $\delta = \delta_H$ in this case, and we have strong control on this from \eqref{ineq4}, as we can directly simplify
    \begin{align*}
        \int (Y_0(y)X(gx))^*Y_0(y')X(gx')\dd g &= \int X(gx)^*Y_0(y)^*Y_0(y')X(gx')\dd g\\
        &= \int X(gx)^*Y_0(y^{-1})Y_0(y')X(gx')\dd g\\
        &= \int X(g)^*Y_0(y^{-1}y')X(gx^{-1}x')\dd g\\
        &= {_0f}'(x^{-1}x', y^{-1}y').
    \end{align*}
    So, from this, we see that
    \begin{equation}
        \delta({_0f}') \leq 2(10+3c)^2\eps^2
    \end{equation}
    and we acquire similar operator bounds as above, finding $(1-2(10+3c)^2\eps^2)I \leq {_0f}'(x,y){_0f}'(x,y)\leq I$ by our above simplification. Since this same bound holds for $|{_0f}|$, we can bound $f_1 = {_0f}'|{_0f}'|^{-1}$ within $2(10+3c)^2\eps^2$ of ${_0f}'$, so we know
    \begin{equation}
        \delta(f_1) \leq 8(10+3c)^2\eps^2.
    \end{equation}
    We can now iterate this process starting from $f_1$ and this defect (noting that $X$ is unchanged as we progress can reduce the growth in coefficient size, if desired).
    \p This clearly has the defect reducing to zero, and since the distance between each successive term is of this form, we merely need to verify the summability of these successive distances. To simplify computations, we note it would be sufficient to verify summability starting from the assumption that $f_0 = {_0f}$, and therefore neglecting terms coming from $X$, as we have seen by a usage of the triangle inequality and Proposition \ref{prop2} that this gives us a function of the appropriate decay near 0 to take this single step.
    \p Making use of this simplification and some easy calculations, the distance from $f_0$ to $f_1$ is bounded by $\eps + 2\eps^2$ and $\delta(f_1)\leq 8 \eps^2$. From $f_1$ to ${_1f}$ is bounded by $8c\eps^2$ with $\delta({_1f}) \leq (8+24c)\eps^2 \leq 32c\eps^2$ when $1 \leq c$, which we can freely assume. So, from here, the distance ${_1f}$ to $f_2$ is bounded by $32c\eps^2 + (32c\eps^2)^2$, so $f_1$ to $f_2$ is bounded by $40c\eps^2 + (32c\eps^2)^2$, and similar patterns carry forward.
    \p The defect $d_{n+1}$ of $f_{n+1}$ is $8(4cd_n)^2 = 128c^2 d_n^2$, and the distance $i_n$ from $f_n$ to $f_{n+1}$ is $5cd_n + 16c^2d_n^2$ after the pattern begins from chaining the distances from $f_n$ to ${_nf}$ to ${_nf}'$ and finally $f_{n+1}$. Since we can determine $d_n = (128c)^{2^{n+1}-1}8^{2^n}\eps^{2^{n+1}}$ we can further compute that
    \begin{equation*}
        i_n = 5cd_n + \frac{1}{8} d_{n+1}
    \end{equation*}
    and so, for sufficiently small $\eps$ relative to $c$, by series comparison we have
    \begin{equation}
        \sum^\infty i_n \leq \sum^\infty 6cd_n \leq 768c\eps\sum^\infty (1024c\eps)^n < \infty,
    \end{equation}
    with the desired decay near 0 quickly shown to hold via evaluation of this geometric series. Since this is uniform in $\mathcal H$, this bounds $F_{G\times H}$, which must be continuous near 0.
\end{proof}

\p Using the estimates of the above theorem and replacing $c\eps$ with $c\eps^s$ where appropriate, as the summation at the end is all that changes, the rate of decay in the defect $4cd_n^s$ of ${_nf}$ and the distance bounds $i_n = 5cd_n^s+(4cd_n^s)^2$ is controlled, assuming $1 \leq c$ and $s \leq 1$, as $2s > 1$. This exponential decay implies summability, yielding the following result.
\begin{corollary}\label{3.13}
    The above theorem holds when $O(\eps)$ is replaced by $O(\eps^s)$ for any $s>\frac 1 2$.
\end{corollary}

\begin{corollary}\label{3.14}
    For a split extension $0 \to G \to Q \to H \to 0$ with associated map $a:H\to \Aut(G)$ such that $G$ is amenable, $H$ is SUS with $F_H\in O(\eps^s)$ for some $s>\frac 1 2$, and $a[H]\leq \Aut(G)$ is an amenable subgroup, we have that $Q$ is SUS.
\end{corollary}
\begin{proof}
    \p Suppose $0\to G \overset\iota\to Q \overset q\to H \to 0$ is a split extension, with section $s:H\to Q$, where $H$ is SUS with $F(\eps)$ as above and $G$ is amenable. For any $x\in Q$, we define $\bar x = s(q(x))$ and $g(x) = x \bar x^{-1} \in \iota[G]$. Much as in Theorem \ref{productpreserve.thm}, we define our partitioning of error as follows:
    \begin{align*}
        \delta_H(f) &= \sup_{x,y\in H} \Vert f(s(xy)) - f(s(x))f(s(y))\Vert\\
        \delta_G(f) &= \sup_{x,y\in G}\Vert f(\iota(xy)) - f(\iota(x))f(\iota(y))\Vert \\
        \delta_s(f) &= \sup_{x\in Q} \Vert f(x) - f(g(x))f(\bar x)\Vert \\
        \delta_c(f) &= \sup_{x\in Q} \Vert f(\bar x)f(g(x)^{\bar x}) - f(g(x))f(\bar x)\Vert
    \end{align*}
    and we note that \eqref{ineqpartition.eq} holds here as well by similar arguments. As above, for $f:Q\to U(\mathcal H)$, we define an integral transform by
    \begin{equation}
        \hat f(x) = \int_G f(g(x)t)f(\bar x)f(t^{\bar x})^*\dd t.
    \end{equation}
    \p Invoking our hypotheses, we have representations $\rho_{0,H}, \rho_G$ on $H, G$ resp. within our bounds, and we define $f_0(x) = \rho_G(g(x))\rho_{0,H}(\bar x)$. For the same reasons reasons as in Theorem \ref{productpreserve.thm}, we know $\delta_G,\delta_H,\delta_s$ are all 0. Further, we note that $\hat{f}_0$ agrees with $f_0$ on all of $\iota[G]$. For $\delta_c$, we see it is 0 as
    \begin{align*}
        \hat f_0(x) &= \int \rho_G(g(x)t)\rho_{0,H}(\bar x) \rho_
        G(t^{\bar x})^*\dd t\\
        &= \int \rho_G(g(x)g(x)^{-1}t)\rho_{0,H}(\bar x) \rho_
        G((g(x)^{-1}t)^{\bar x})^*\dd t\\
        &= \int \rho_G(t)\rho_{0,H}(\bar x) \rho_
        G((t^{-1}g(x))^{\bar x})\dd t\\
        &= \bigg(\int \rho_G(t)\rho_{0,H}(\bar x) \rho_
        G(t^{\bar x})^* \dd t\bigg) \rho_G(g(x)^{\bar x})\\
        &= \hat f_0(\bar x) \hat f_0(g(x)^{\bar x}).
    \end{align*}
    \p From this, we can immediately see $\delta(\hat f_0) = \delta_H(\hat f_0)$, further we can bound $\Vert \hat f_0 - f_0\Vert \leq \delta(f_0)$ by one application of $\delta_c$, and for any function $F$ we have $\Vert \hat F - F\Vert \leq 3\delta(F)$ by similar applications of triangle inequality. To attempt to control $\delta_H$ as in \eqref{ineq4}, we act similarly: we expand out for $x,y \in H$
    \begin{align*}
        &\int (f_0(x) - f_0(gx)f_0((g^{-1})^x))^*(f_0(y) - f_0(gy)f_0((g^{-1})^y)) \dd g\\
        =& \int f_0(x)^*f_0(y) - f_0(x)^*f_0(gy)f_0((g^{-1})^y)\\
        &- f_0((g^{-1})^x)^*f_0(gx)^*f_0(y) + f_0((g^{-1})^x)^*f_0(gx)^*f_0(gy)f_0((g^{-1})^y) \dd g\\
        =& f_0(x)^*f_0(y) - f_0(x)^*\hat f_0(y)\\
        &- \hat f_0(x)^* f_0(y) + \int f_0((g^{-1})^x)^*f_0(gx)^*f_0(gy)f_0((g^{-1})^y)\dd g\\
        =& (\hat f_0(x)-f_0(x))^*(\hat f_0(y)-f_0(y))-\hat f_0(x)^*\hat f_0(y)\\
        &+ \int f_0((g^{-1})^x)^*f_0(gx)^*f_0(gy)f_0((g^{-1})^y)\dd g
    \end{align*}
    which similarly has norm control of $\delta(f_0)^2$, and hence 
    \begin{equation}\label{theint}
        \bigg\Vert \int f_0((g^{-1})^x)^*f_0(gx)^*f_0(gy)f_0((g^{-1})^y)\dd g - \hat f_0(x)^*\hat f_0(y)\bigg\Vert \leq 2 \delta(f_0)^2.
    \end{equation}
    \p To finalize the proof, we first manipulate the integral from \eqref{theint} in a formal computation where we assume that we can freely conjugate the variable of integration
    \begin{align*}
        &\int f_0((t^{-1})^x)^*f_0(tx)^*f_0(ty)f_0((t^{-1})^y)\dd t \\
        =& \int \rho_G(t^x) \rho_{0,H}(x^{-1})\rho_{G}(t^{-1}t)\rho_{0,H}(y)\rho_G(t^y)^*\dd t\\
        =& \int \rho_G(t^x) \rho_{0,H}(x^{-1}y)\rho_G(t^y)^*\dd t\\
        =& \int \rho_G(t) \rho_{0,H}(x^{-1}y)\rho_G(t^{x^{-1}y})^*\dd t\\
        =& \hat{f}_0(x^{-1}y).
    \end{align*}
    \p Though this does assume we have invariance of $\int\dd t$ under conjugation, we can similarly utilize such a manipulation and invariance under inversion of $t$ (which can always be done over amenable $G$) to prove $\hat{f}_0(x)^* = \hat{f}_0(x^{-1})$. 
    \p These manipulations rely on an invariance of our measure under the action of $H$ factoring through a map $a:H\to \Aut(G)$ induced by conjugation in $Q$, but we know that the image $a[H]$ is amenable by hypothesis. Utilizing this, we define the following mean:
    \begin{equation*}
        \mu^{*\nu}(X) = \int_{a[H]} \mu(a(x)[X])\dd \nu(a(x))
    \end{equation*}
    where $\mu$ refers to our mean on $G$ and $\nu$ to our mean on $a[H]\leq \Aut(G)$, which exists by hypothesis. That this is an invariant mean under the $G\rtimes a[H]$ action is trivial to verify by utilizing that $a(x)[gX] = a(x)(g)\cdot a(x)[X]$ and similar such expressions, and we can then apply the same estimates and convergence from the end of the proof of Theorem \ref{productpreserve.thm}. This yields the desired representation, and proves the bounds on SUS for $G\rtimes_a H \cong Q$, much as in Corollary \ref{3.13}
\end{proof}

\subsection{Remark}\label{3.17}
    We note that the above theorems do not make essential use of the SUS property but rather only stability estimates particular to $U(\mathcal H)$, which implies the theorem by working fiberwise. Because of this, these proofs imply a marginally stronger result about preserving relevant asymptotic estimates at each dimension individually even without the uniformity across all by strong Ulam stability.

\section{Inductivity and Logic}

\p In this section, we prove a few results in pursuit of inductivity of the class of SUS groups, including some additional lemmas on existence of some invariant measures on ordinals which respect the order structure. This section also contains the logical relevance of such results, as well as Theorem \ref{productpreserve.thm}, in terms of existential closures as well as quantitative control on estimates. We first begin with a definition of what we mean by inductivity. 
\begin{definition}[Inductivity]
    We say a class of structures $\mathcal{K}$ is \emph{inductive} if, for any linear order $I$ and any chain of structures $(M_i)_{i\in I}$ such that $i\leq j$ implies $M_i \leq M_j$ as substructures and each $M_i$ is in $\mathcal{K}$, we have that $M=\bigcup_{i\in I} M_i$ is in $\mathcal{K}$.
\end{definition}

To that end, we note that for any linear order $I$, we can replace it with a cofinal well order $\kappa$, so it is sufficient to prove this property over regular cardinals. Further, since the proper inclusions will either be cofinal or the chain will stabilize, in which case the result is trivial, we can assume going forward that $i < j$ implies $G_i < G_j$.
\p Given a sequence of ``nice" maps $f_i:G_i \to B(\mathcal H)$ with a uniform bound, the most obvious way to get a ``nice" map $f:G \to B(\mathcal H)$ is through $f = \int f_i \dd\mu$ for an appropriate measure $\mu$ after extending each $f_i$ over $G$. The obvious question at that point is: how poorly behaved can this integral be despite our best efforts? That is, how tightly can we control defect terms in terms of those of the $f_i$ and how the defects vary as we increase $i$?

\begin{definition}[Cofinal measure]
    We call a finitely additive measure $\mu$ on $\kappa$ cofinal if $\mu(\{i|i \leq \alpha\})=0$ for every $\alpha < \kappa$, in particular the case where $\mu$ is positive.
\end{definition}

\begin{proposition}\label{asympdef}
    Given a $\kappa$-sequence of unitary representations $f_i:G_i \to U(\mathcal H)$ for a $\kappa$-chain $(G_i)$ and a cofinal probability measure $\mu$ on $\kappa$, we have that
    \begin{equation*}
        \delta\bigg(\int f_i\dd\mu\bigg) \leq \inf_{\alpha<\kappa}\sup_{\alpha < i<j<\kappa}\sup_{x\in G_i} \Vert f_i(x)-f_j(x)\Vert.
    \end{equation*}
\end{proposition}

\begin{proof}
    Let $f(x) = \int f_i(x)\dd\mu$, where $f_i(x) = 0$ for $x \not\in G_i$, which is safe as $\mu$ is cofinal, and hence $\mu(\{i|f_i(x)=0\}) = 0$. Suppressing $\mu$ to keep track of variables, we note that
    \begin{align*}
        \Vert f(xy)-f(x)f(y)\Vert &= \bigg\Vert \int_\kappa f_i(xy)\dd i - \bigg(\int_\kappa f_i(x)\dd i\bigg)\bigg(\int_\kappa f_j(y)\dd j\bigg)\bigg\Vert\\
        &= \bigg\Vert \int_\kappa f_i(xy)-f_i(x)\int_\kappa f_j(y)\dd j\dd i \bigg\Vert\\
        &= \bigg\Vert \int_\kappa\int_\kappa f_i(x)f_i(y)-f_i(x)f_j(y)\dd j \dd i\bigg\Vert\\
        &= \bigg\Vert \int_\kappa\int_\kappa f_i(x)\big(f_i(y)-f_j(y)\big)\dd j \dd i\bigg\Vert\\
        &\leq \int_\kappa \int_\kappa \big\Vert f_i(x)\big(f_i(y)-f_j(y)\big)\big\Vert\dd j \dd i\\
        &\leq \int_\kappa \int_\kappa \Vert f_i(y)-f_j(y)\Vert \dd j \dd i
    \end{align*}
    \p To conclude the desired bound, we then note that the integrand is immediately bounded by taking a supremum across all of $G = \bigcup_{i<\kappa}G_i$ and for any pair $x,y\in G$ we can find an $i<\kappa$ such that $x,y\in G_i$, so it is sufficient to look at subgroups listed in our chain. Further, the 0 evaluation off of the current subgroup is of negligible impact, as $\mu$ is cofinal, and for the same reason we can see that $\int_{\kappa} g(i)\dd i = \int_{\kappa\setminus\{z|z\leq \beta\}} g(i)\dd i$ for any $\beta$. We conclude the proof by noting that, applying this for any $\alpha <\kappa$, we have
    \begin{equation*}
        \int_\kappa \int_\kappa \Vert f_i(y)-f_j(y)\Vert \dd j \dd i = \int_{\kappa\setminus\{z|z\leq\alpha\}} \int_{\kappa\setminus\{z'|z'\leq i\}} \Vert f_i(y)-f_j(y)\Vert \dd j \dd i
    \end{equation*}
    Applying the $\infty$-norm bounds guaranteed by integration to this integrand, and that we have the bound on $\delta(f)$ above for any $\alpha<\kappa$, we complete the proof.
\end{proof}
\p By the above lemma, it is immediate that controlling asymptotically wild variation behavior in representations approximating a given function on each named subgroup is central to inductivity of SUS in the presence of a continuous uniform bound on modulus functions. We note that, by a result of Kazhdan, every amenable group has a bound of $F_G(\eps)\leq 2\eps$ for small enough $\eps$, and our Theorem \ref{productpreserve.thm} relies on similar decay asymptotics that hold in all known examples.

\p Performing some formal computations under the assumption of several invariance and limit interchange properties of $\int \dd t$, which we know fail often and are why particular measures are so well studied, we manipulate the ``oscillation-like" $\iint f(i)f(j)\dd i\dd j$ expression from Proposition \ref{asympdef}. One very handy tool we had used before is translation invariance, so let's try $\iint f(i+t)f(j+t)\dd j\dd i$, which changes nothing. That's a constant function in $t$, so we can freely integrate $\iiint f(i+t)f(j+t)\dd j \dd i \dd t$. We then want to swap the integration order, so as to move $t$ to the innermost integral. Since we can assume $t \gg i, j$ by cofinality of our mean, we get that our expression is equal to $\iiint f(t)^2\dd t \dd j \dd i$ by ordinal arithmetic, which ``gives" us 0 oscillation, under these formal computations and symbol manipulations.
\p The issue there is we do not know if we can perform these kinds of manipulations, but the following results show us how we can guarantee at least some of these properties:

\begin{proposition}\label{measureexistence}
    There exists a cofinal probability measure on every ordinal of the form $\omega^\alpha$ which is right translation invariant for ordinals of lesser exponent. In particular, it is right translation invariant for the infinite cardinals $\kappa$ (considered as initial ordinals).
\end{proposition}

\begin{proof}
    \p First, as motivation, we recreate the standard proof of how we can have such a measure on $\bZ$ and hence $\bN=\omega$. Take a nonprincipal ultrafilter $\mathcal{U}$ on $\omega$, and define $\mu_1(A)$ as \begin{equation*}
        \mu_1(A)=\lim_{\mathcal U(n)} \frac{|A\cap \{0,\dots, n-1\}|}{n}
    \end{equation*}
    As can be easily verified by directly computing additivity for disjoint sets and seeing that $\mu_1(A \Delta (A+1))$ computation gives $\lim_{\mathcal U(n)} \frac 2 n = 0$, this gives us an addition invariant finitely additive probability measure on $\omega$.
    \p Next, for each ordinal $\alpha$, we want to define a right translation invariant cofinal probability measure $\mu_\alpha$ on $\omega^\alpha$. We proceed by induction, assuming we have $\mu_\alpha$ to construct $\mu_{\alpha+1}$. Since $\omega^{\alpha+1}=\omega^\alpha\cdot\omega$, we can split it into $\omega$ many $\omega^\alpha$-segments; for any set of concern $A\subset \omega^{\alpha+1}$, define $A_n$ as the subset contained within the $n$th $\omega^\alpha$ segment. 
    \p We define $\mu_{\alpha+1}$ as follows:
    \begin{equation*}
        \mu_{\alpha+1}(A)=\lim_{\mathcal U(n)} \frac{\sum_{i=0}^n \mu_{\alpha}(A_i)}{n}
    \end{equation*}
    Since $\mu_\alpha$ is translation invariant for $x < \omega^\alpha$, and as $(A+x)_n = A_n+x$, this measure is invariant under right translation. Similarly to the case of $\omega$ above, it is trivial to compute and verify this is indeed a finitely additive measure. It is cofinal as any initial segment is contained in some $\omega^\alpha$ segment, and hence the ultralimit is of the form $\frac c n\to 0$ for some constant $c$ for sufficiently large $n$.
    \p By existence of Cantor normal form, it is sufficient to show translation invariance under $\omega^\alpha$ to show it for all feasible ordinals; however, $(A+\omega^\alpha)_n = A_{n+1}$, and hence the difference in measure is bounded by $\frac{\mu_{\alpha}(A_0) +\mu_{\alpha}(A_{n+1})}{n} \leq \frac 2 n \to 0$. So it is translation invariant for all ordinals whose Cantor normal form has smaller exponent.
    \p To complete the proof, we define $\mu_\beta$ for limit ordinals $\beta$ by
    \begin{equation*}
        \mu_\beta(A) = \lim_{\mathcal U_{\beta}(\alpha)} \mu_\alpha(A\cap \omega^\alpha),
    \end{equation*}
    where $\mathcal{U}_\beta$ is an ultrafilter extending the tail filter on $\omega^\beta$. Since this is translation invariant under $\omega^\alpha$ for each $\alpha<\beta$, and the $\omega^\alpha$ are cofinal in $\omega^\beta$ guaranteeing cofinality of the measure, this gives us our desired measure for every ordinal. The last part of the theorem follows from the fact that $\omega_\alpha= \omega^{\omega_\alpha}$ for each initial ordinal (= cardinal).
\end{proof}

\p We note that this measure and ordinal arithmetic implies $\iint f(i+j)\dd i \dd j = \iint f(i+j)\dd j \dd i = \int f(t)\dd t$, but this is insufficient to show the desired integral exchange due to interference between the variables, hence failing to fully imply inductivity. Some results on interchange of integrals between finitely additive measures are known, see \cite{Sin}, however the results of that work appear to be too coarse to give a good theory for interchange with respect to the same measure as the DLC condition in that work is too strong to ensure.

\p It is also reasonable to note that preservation of SUS in products $G\times H$ and inductivity among all SUS groups implies that we have a uniform bound on the modulus of stability.

\begin{proposition}
    If the class of SUS groups is inductive and closed under finite products, there exists a continuous $F(\eps):[0,2]\to [0,2]$ with $F(0)=0$ such that, for any SUS group $G$, we have $F_G(\eps)\leq F(\eps)$
\end{proposition}

\begin{proof}
    If no such bound existed, take a sufficiently long $\kappa$-sequence of SUS groups $(G_i)$ such that $\sup_{i<\kappa} F_{G_i}$ is discontinuous at 0. Since for every $i < \omega$, we have $H_i= \bigoplus_{j<i} G_j$ is SUS, by inductivity we have that $H_\omega$ is SUS. Continuing similarly by ordinal induction, applying product preservation and inductivity as requisite, we find that $H_\kappa$ is SUS.
    \p Applying our hypothesis that the moduli of stability are not uniformly bounded by a continuous function, we find for each $G_i$ a map $f_i$ of small defect requiring sufficiently large distance to representations and extend it to $H_\kappa$ in the obvious fashion ignoring all other factors. This has the same distance to a representation, much as how $F_{G\times H} \geq F_G$, and as such we have that $F_{H_\kappa}$ cannot be continuous at 0, contradicting our assumption of product preservation and inductivity.
\end{proof}

\p Among finite dimensional unitary groups, however, we can prove directly prove inductivity in a strong quantitative form by making use of the local compactness in finite dimensional spaces. This can be done fiberwise, even, giving the following theorem:

\begin{theorem}\label{4.6}
    Given a proper $\kappa$-sequence of subgroups $(G_i)$ which are all uniformly $(U(d), \Vert\cdot\Vert_s)$ stable for any norm $\Vert\cdot \Vert_s$ such that $F(\eps) = \sup_i F^{(d)}_{G_i}(\eps)$ is continuous at 0, we have that $G = \bigcup_i G_i$ is uniformly stable with $F^{(d)}_G(\eps)\leq F(\eps)$.
\end{theorem}

\begin{proof}
    Take a function $f:G\to U(d)$ of defect $\eps$ and choose representations $\rho_i$ within $F(\eps)$ of $f|_{G_i}$ for each $i<\kappa$, then extend these to functions on all of $G$ by making it uniformly $1$ outside $G_i$. Since $U(d)$ is compact, so is $U(d)^G$, so there is a limit point $\rho$ to which a \emph{cofinal} collection of the $\rho_i$ converge to pointwise. Since for each $x\in G$ we know that eventually $\Vert f(x) - \rho_i(x)\Vert \leq F(\eps)$, we know that $\Vert f(x) - \rho(x)\Vert \leq F(\eps)$ as that defined a closed set, and, similarly, continuity of multiplication by all norms being equivalent would imply $\rho(xy) = \rho(x)\rho(y)$ for each $x,y\in G$. This gives the desired representation with a preserved bound on our modulus of stability.
\end{proof}

\subsubsection*{Remark}
    The manner of choice of $\rho$ can be performed uniformly (up to choices of $(\rho_i)$, of course) among all $f$ and proper $\kappa$-sequences $(G_i)$ by choice of a cofinal ultrafilter $\mu$ on $\kappa$ and defining $\rho = \int \rho_i \dd\mu(i)$, where the integration in this case can also be seen in the sense of an ultralimit in the norm topology. This uniformity does not depend on the quantitative bounds invoked in the choice of the proper $\kappa$-sequence, nor on the dimension $d$, but makes essential use of the finiteness of $d$ (and hence compactness) to remain as a representation over every named subgroup, which eventually covers $G$.

\p As a trivial consequence of classical model theoretic results, we derive the following corollary after recalling a definition, though we mention it explicitly to emphasize the quantitative nature of both the closures and the classes among which it is existentially closed.

\begin{definition}[Existentially closed]
    For a language $L$, an object $X$ is existentially closed with respect to a class $\mathcal K$ of $L$-structures if for any object $K \in \mathcal K$ which extends $X \leq K$, and any $\exists_1$ sentence $\phi$ in $L(X)$, we have that $K \models \phi$ implies that $X \models \phi$. Equivalently, for any $X \leq K \in \mathcal K$, there is an ultrafilter $\mu$ where $X \leq K \leq X^\mu$ such that the composite embedding $X \leq X^\mu$ is elementary.
\end{definition}

\begin{corollary}\label{4.9}
    Given any continuous decreasing function $F(\eps):[0,2] \to [0,2]$ such that $F(0) = 0$, any subset $A\subset \bN$, and any group $G$ such that $F^{(d)}_G(\eps) \leq F(\eps)$ for all $d\in A$, there exists a group $G' \geq G$ which is existentially closed among groups satisfying these stability estimates \emph{and} $F^{(d)}_{G'}(\eps)\leq F(\eps)$ for all $d\in A$.
\end{corollary}

We remark the above theorem and corollary apply essentially unchanged for any family $\mathscr{G}$ of compact metric groups in the appropriate fashion (where $A\subset \bN$ is instead replaced by $A\subset \mathscr{G}$), but the finite dimensional unitaries are of particular interest, of course. For further information on this, \cite{FR1} utilizes and proves properties around Ulam stability or the further generalization of uniform $\mathfrak U$ stability, including relevance towards determining amenability of the Thompson group or other self-similar groups.

\subsubsection*{Remark}
    Using a result of Francesco Fournier-Facio and Bharatram Rangarajan, Theorem 1.3 in \cite{FR1}, we can see that for any countable group $G$ and constant $c$ sufficiently large, there is a countable group $G' \geq G$ which is Ulam stable (or any other choice of uniformly $\mathfrak{U}$ stable to bound) with $F^{fd}_{G'} \leq c\eps$ that is existentially closed among  groups with such bounds. Because of how ``easy" it is to grant this kind of stability, one might ask just how broadly or uniformly we can create ``existential closures" of groups. If one has some form of quantitative control on the linear estimates arising, namely such that you can find a $c$ large enough to where enough groups $H \geq G$ have some infinite amenable $\Gamma$ with $H\wr \Gamma$ bounded in estimate by $c$, then the group $G'$ so constructed could be chosen to be existentially closed among \emph{all} groups. 

To that end, we ask the following question:
\begin{problem}
    For a given countable group $G$, what are the optimal Lipschitz constants with respect to Ulam stability of the groups $G\wr A$ for $A$ countably infinite amenable? Conversely, what are the optimal constants possible for a fixed $A$ varying $G$ across all countable groups?
\end{problem}

\p We conjecture that some control on the linear estimates guaranteed for fixed countably infinite amenable $A$ in \cite{FR1} for wreaths may be controllable via some embeddings $(X\times Y) \wr A \leq (X\wr A)\times (Y \wr A) \leq (X\times Y)\wr A^2$ or otherwise finding some appropriate coamenable embeddings, but we have not yet found the correct estimates for coamenable subgroups or coamenable embeddings to best make use of for this purpose. While some analogues of the estimates presented here can be adapted to a ``coamenable convolution" setting, some difficulties appear to arise in the correct defect bounding so as to ensure convergence to representations. Further, the dependence on choice of transversal type maps appears to be high, making these integral limiting processes very sensitive to the structure of the embedding.

\section{Further Directions}

\p The theorems here depend significantly on decay estimates of the modulus near 0, but the issue is that we lack many theorems which control the decay asympotically. While \cite{FR1} does prove ways to construct many countable groups which have $O(\eps)$ decay at 0, \cite{GLMR} provides asymptotic cohomology conditions for $O(\eps)$ decay and Proposition 1.0.6 describes $U(1)$-stability, and \cite{BOT} in Section 5 does prove in Theorem 5.1 that $SL(n,A), n\geq 3$, is Ulam stable with Lipschitz decay, there is a dearth of control in general. For this reason, we find it reasonable that a further inquiry into the decay asymptotics which can occur in various stability conditions and raise the following conjecture.

\begin{conjecture}\label{conj5.1}
    Every Strongly Ulam Stable group $G$ has $O(\eps)$ decay at 0 of the SUS modulus $F_G$.
\end{conjecture}
\p This conjecture is a weaker form of the question of existence of non-amenable SUS groups and could be considered as related to the question posed in \cite{BOT} right after Theorem 3.1 of that paper, though that was asking about amenable groups which have a known Lipschitz bound from these theorems. While this is obviously inspired by the idea that they may already be amenable, during the writing of this paper we have noticed that methods of gluing approximate representations together does appear to prevent significant reductions in Lipschitz constants by passing to higher dimensions. In particular, a relatively easy observation:

\begin{proposition}\label{prop5.2}
    Let $G$ be SUS, and suppose $L'\eps^{s}\geq F_G(\eps) > L\eps$ for $s>\frac{1}{2}$ and $\eps>0$ sufficiently small. Then, given $0<c<L$, for all sufficiently small $\eps >0$ and $f_i: G\to U(\mathcal H_i)$ with $\delta(f_i)\geq \eps$ and $d(f_i,\mathrm{Hom}(G,U(\mathcal H_i)))>L\eps$ for $i=1,2$, we have that $d(f_1\oplus f_2, \mathrm{Hom}(G, U(H_1\oplus H_2))) > c\eps$.
\end{proposition}

\begin{proof}
    \p Fix some $0<c<L$, let $\mathcal H = \mathcal H_1\oplus \mathcal H_2$, and suppose $\rho:G\to U(\mathcal H)$ is a representation such that $\Vert \rho - f_1\oplus f_2\Vert_\infty \leq c\eps$. Making use of the orthogonal projections $p_i:\mathcal H\to \mathcal H_i$, we note a block decomposition of $\rho$ as
    \begin{equation}
        \rho = \begin{bmatrix}\rho_1 & C_{21} \\ C_{12} & \rho_2\end{bmatrix}.
    \end{equation}
    From this, and that the norm of the operator $A\mapsto p_1Ap_1+p_2Ap_2$ is 1, we note that $\Vert \rho_1\oplus \rho_2 - f_1\oplus f_2\Vert_\infty \leq c\eps$. By triangle inequality and assumptions, we similarly know that $\rho_1$ has distance greater than $(L-c)\eps$ from a representation in $U(\mathcal H_1)$, and hence $F_G(\delta(\rho_1)) > (L-c)\eps$. In particular, as $F(\nu) > L\nu$ for sufficiently small $\nu$, we can see that $L'\delta(\rho_1)^s \geq F(\delta(\rho_1)) > L\delta(\rho_1)$ when $\eps$ is sufficiently small. So, we know that 
    \begin{equation}\label{eq5.4}
        \delta(\rho_1) > \bigg(\frac{L-c}{L'}\eps\bigg)^{\frac{1}{s}}.
    \end{equation}
    \p We note that $\rho(xy)=\rho(x)\rho(y)$, and by utilizing our matrix block form, we can find that
    \begin{equation}\label{eq5.5}
        \rho(x)\rho(y) = \begin{bmatrix}
            \rho_1(x)\rho_1(y) + C_{21}(x)C_{21}(y) & \rho_1(x)C_{21}(y)+C_{12}(x)\rho_2(y)\\
            C_{12}(x)\rho_1(y)+\rho_2(x)C_{12}(y) & C_{12}(x)C_{21}(y)+\rho_2(x)\rho_2(y)
        \end{bmatrix}.
    \end{equation}
    Utilizing \eqref{eq5.5} above in tandem with \eqref{eq5.4}, we see that $\Vert C_{12}C_{21}\Vert_\infty > \big(\frac{L-c}{L'}\eps\big)^{1/s}$, so one of these terms has norm greater than $\big(\frac{L-c}{L'}\eps\big)^{1/2s}$ by submultiplicativity. However, by the triangle inequality utilizing that $\Vert \rho_1\oplus \rho_2 - f_1\oplus f_2\Vert_\infty \leq c\eps$ and $\Vert \rho - f_1\oplus f_2\Vert_\infty \leq c\eps$, we see that 
    \begin{equation}
        2c\eps \geq 2\Vert \rho - \rho_1\oplus \rho_2\Vert_\infty \geq \Vert C_{12}\Vert_\infty + \Vert C_{21}\Vert_\infty > \bigg(\frac{L-c}{L'}\eps\bigg)^{\frac{1}{2s}}
    \end{equation}
    which presents a contradiction for sufficiently small as $2s>1$ and hence the rightmost term is not Lipschitz at 0, so there is a point at which sufficiently small epsilon violate this inequality. Namely, the claim of the theorem holds when
    \begin{equation}
        \eps \leq \bigg(\frac{L-c}{(2c)^{2s}L'}\bigg)^{\frac{1}{2s(1-2s)}} 
    \end{equation}
\end{proof}

\p We note that in the above proposition, we only use the defect of the top left corner of $\rho$, which would allow gluings that do not permit a reduction in Lipschitz constant data in infinite direct sums (under one way to complete them). In particular, if we have a SUS group $G$ with H\"older exponent $s > \frac{1}{2}$, we can choose a sequence $f_i$ which captures that $F_G(\eps) \gg n\eps$ for sufficiently small $\eps < \eps_n$ if $s < 1$, or otherwise capture the Lipschitz constant $L$ of $F_G$ if $s=1$. That is, the germ of $F_G$ at 0 can be encoded as a single approximate representation over $G$ along with the projection to the tails of the sequence when our stability property is uniform across a class of spaces closed under infinite direct sums.
\p We remark further that none of the above relies on strong Ulam stability in particular, but only a control on distance from representations utilizing a defect term and our stability estimates over $\mathcal H_1$ to bound the distance from 0 of the off-diagonal terms. This, therefore, applies to a much broader class of stability properties over submultiplicative norm, but we state it over strong Ulam stability as being the topic of this paper.
\p Seeing as there is strong quantitative controls one can perform ``fiberwise" with stability properties once we have decay estimates near 0, and this behavior is as expected when we have H\"older exponent $s >\frac{1}{2}$, we ask whether one can find stability properties where this H\"older condition fails, and we posit that this strong quantitative control on loss of Lipschitz constants in these diagonal direct sums is good evidence towards controllability of these decays. Even in the absence of bad decays, having explicit forms on loss of constants under these transformations is of interest to control deviations or find analogues of error corrections.

\bibliography{MSSUSGroups}
\bibliographystyle{alpha}

\end{document}